\documentclass{amsart}

\makeatletter
\def\subsection{\@startsection{subsection}{2}%
  \z@{.5\linespacing\@plus.7\linespacing}{.5\linespacing}%
  {\normalfont\bfseries}}
\makeatother

\makeatletter
\def\@defaultbiblabelstyle#1{[#1]}
\makeatother

\makeatletter
\def\@setauthors{%
  \begingroup
  \def\thanks{\protect\thanks@warning}%
  \trivlist
  \centering\footnotesize \@topsep30\p@\relax
  \advance\@topsep by -\baselineskip
  \item\relax
  \author@andify\authors
  \def\\{\protect\linebreak}%
  \authors%
  \ifx\@empty\contribs
  \else
    ,\penalty-3 \space \@setcontribs
    \@closetoccontribs
  \fi
  \endtrivlist
  \endgroup
}
\def\@settitle{\begin{center}%
  \baselineskip14\p@\relax
    \bfseries
  \@title
  \end{center}%
}
\makeatother

\usepackage{amssymb}
\usepackage[colorlinks=true, citecolor=blue]{hyperref}
\usepackage{upref}
\usepackage{tikz}
\usetikzlibrary{calc}

\newtheorem{theorem}{Theorem}[section]
\newtheorem{lemma}[theorem]{Lemma}
\newtheorem{proposition}[theorem]{Proposition}
\newtheorem{corollary}[theorem]{Corollary}

\theoremstyle{definition}
\newtheorem{definition}[theorem]{Definition}

\theoremstyle{remark}
\newtheorem{remark}[theorem]{Remark}

\numberwithin{equation}{section}
\allowdisplaybreaks[4]

\begin{document}

\title{On exterior powers of reflection representations, II}

\author{Hongsheng Hu}
\address{School of Mathematics, Hunan University, Changsha 410082, China}
\email{huhongsheng@hnu.edu.cn}
\urladdr{\url{https://huhongsheng.github.io/}
         \newline \indent {\rmfamily {\itshape ORCID}:} \href {https://orcid.org/0000-0002-1850-056X} {0000-0002-1850-056X}}



\date{January 28, 2025}

\begin{abstract}
Let $W$ be a group endowed with a finite set $S$ of generators.
A representation $(V,\rho)$ of $W$ is called a reflection representation of $(W,S)$ if $\rho(s)$ is a (generalized) reflection on $V$ for each generator $s \in S$.
In this paper, we prove that for any irreducible reflection representation $V$, all the exterior powers $\bigwedge ^d V$, $d = 0, 1, \dots, \dim V$, are irreducible $W$-modules, and they are non-isomorphic to each other.
This extends a theorem of R.\ Steinberg which is stated for Euclidean reflection groups.
Moreover, we prove that the exterior powers (except for the 0th and the highest power) of two non-isomorphic reflection representations always give non-isomorphic $W$-modules.
This allows us to construct numerous pairwise non-isomorphic irreducible representations for such groups, especially for Coxeter groups.
\end{abstract}

\maketitle

\setcounter{tocdepth}{2}
\tableofcontents

\section{Introduction} \label{sec-intro}

\subsection{Overview}

In \cite[Section 14]{Steinberg1968}, R.\ Steinberg proved a theorem stating that the exterior powers of the irreducible reflection representation of a Euclidean reflection group are again irreducible and pairwise non-isomorphic (see also \cite[Chapter V, Section 2, Exercise 3]{Bourbaki2002}).
For Weyl groups, the exterior powers of the standard reflection representation are well studied (see, for example, \cite{CIK71, GP00, Henderson10, Sommers11}).

The proof of Steinberg's theorem relies on the existence of an inner product which stays invariant under the group action.
In a previous paper \cite{Hu23-ext-pow}, the author extended Steinberg's result to a more general context where the inner product may not exist.
Let $W$ be a group and $S= \{s_1, \dots, s_k\}$ be a set of generators of $W$.
We say a representation $\rho: W \to \operatorname{GL}(V)$ is a reflection representation of $(W,S)$ if each of the generators $s_i$ acts by a generalized reflection, and denote by $\alpha_i$ the chosen reflection vector (see Subsection \ref{subsec-refl} for related notions).
The main theorem in \cite{Hu23-ext-pow} reads:

\begin{theorem}[{\cite[Theorem 1.2]{Hu23-ext-pow}}] \label{thm-pre}
  Let $(V,\rho)$ be an $n$-dimensional irreducible reflection representation of $(W,S)$ over a field $\mathbb{F}$ of characteristic 0, with reflection vectors $\alpha_1, \dots, \alpha_k$.
  Suppose
  \begin{equation}\label{eq-cond-remove}
    \text{for any two indices $i,j$, $s_i \cdot \alpha_j \ne \alpha_j$ if and only if $s_j \cdot \alpha_i \ne \alpha_i$}.
  \end{equation}
  Then the $W$-modules $\{\bigwedge^d V \mid 0 \le d \le n\}$ are irreducible and pairwise non-isomorphic.
\end{theorem}

As pointed out in \cite{Hu23-ext-pow}, usually there is no $W$-invariant bilinear form on the reflection representation, so that our result is not a trivial generalization.

The first aim of this paper is to show that the assumption \eqref{eq-cond-remove} can be removed:

\begin{theorem} \label{thm-main1}
  Let $(V,\rho)$ be an $n$-dimensional irreducible reflection representation of $(W,S)$ over a field $\mathbb{F}$ of characteristic 0.
  Then the $W$-modules $\{\bigwedge^d V \mid 0 \le d \le n\}$ are irreducible and pairwise non-isomorphic.
\end{theorem}

The readers may find that the proof of Theorem \ref{thm-main1} is similar to that of Theorem \ref{thm-pre} (\cite[Theorem 1.2]{Hu23-ext-pow}).
The proof here simplifies the proof in \cite{Hu23-ext-pow} a little bit.
See Section \ref{sec-main1} for more details.

The major contribution of this paper is the second main result, stating that
the exterior powers of two different reflection representations are also different.
To be precise, we have

\begin{theorem} \label{thm-main2}
  Let $(V_\iota, \rho_\iota)$, $\iota = 1, 2$, be two irreducible reflection representations of $(W,S)$ over a field $\mathbb{F}$ of characteristic 0, with dimensions $n_1$ and $n_2$ respectively.
  Suppose $\bigwedge^{d_1} V_1 \simeq \bigwedge^{d_2} V_2$ as $W$-modules for some integers $d_1, d_2$ with $1 \le d_\iota \le n_\iota - 1$ $(\iota = 1, 2)$.
  Then $d_1 = d_2$, $n_1 = n_2$, and $V_1 \simeq V_2$ as $W$-modules.
\end{theorem}

\begin{remark}
  Note that $\bigwedge^0 V$ is the one-dimensional $W$-module with trivial $W$-action. While $\bigwedge^n V$ carries the one-dimensional representation $\det \circ \rho$ for any $n$-dimensional representation $(V,\rho)$ of $W$, and different $\rho$'s might share the same determinant $\det \circ \rho$ (for example, if each generator $s_i$ is of order two, then $\det \circ \rho (s_i) = -1$ for any reflection representation $(V,\rho)$ and any $i$).
  Thus, in Theorem \ref{thm-main2} the range $1 \le d_\iota \le n_\iota -1$ is the best we can expect.
\end{remark}

By combining the results in Theorems \ref{thm-main1} and \ref{thm-main2}, immediately we have the following corollary, which allows us to construct numerous pairwise non-isomorphic irreducible representations for the group $W$.

\begin{corollary}
  Suppose we have a family of irreducible reflection representations $\{V_i \mid i \in I\}$ of $(W,S)$.
  Then
  $\{\bigwedge^d V_i \mid i \in I,  1 \le d \le \dim V_i - 1\}$
  is a family of simple $W$-modules, and they are pairwise non-isomorphic.
\end{corollary}

\subsection{Motivation and application}

The motivation of this work (as well as the previous \cite{Hu23-ext-pow}) comes as follows.
Suppose $W$ is a Coxeter group with the finite set $S$ of defining generators.
In another paper \cite{Hu23} by the author, all the reflection representations (over $\mathbb{C}$) of $(W,S)$ are determined.
The most essential thing in this process  is the classification of isomorphism classes of the so-called \emph{generalized geometric representations} (that is, those reflection representations admitting  a basis formed by the reflection vectors).
In \cite{Hu23}, such representations are classified using the characters of the first integral homology group of simple graphs which are closely related to the Coxeter graph.
Moreover, ``most'' of them are irreducible.
While if a generalized geometric representation is reducible, then it has a semisimple quotient, each of whose direct summand is an irreducible reflection representation of some parabolic subgroup.
Therefore, the results in this paper are applicable, and then we obtain a large class of irreducible representations which are non-isomorphic to each other.

For example, if $(W,S)$ is the affine Weyl group of type $\widetilde{A}_n$, that is,
\[W = \langle s_0, s_1, \dots, s_n \mid s_i^2 = (s_is_{i+1})^2 = e, \forall i = 0,1, \dots, n\rangle\]
(regard $n+1$ as $0$), then the Coxeter graph is a cycle.
The corresponding first homology group with integral coefficients is isomorphic to $\mathbb{Z}$, and its characters are parameterized by $\mathbb{C}^\times$, and so are the generalized geometric representations.
For $x \in \mathbb{C}^\times$, the corresponding generalized geometric representation, denoted by $V_x$, is an ($n+1$)-dimensional $\mathbb{C}$-vector space with basis $\{\alpha_0, \alpha_1, \dots, \alpha_n\}$, and the $W$-action can be defined by
\begin{align*}
  s_i \alpha_i & = - \alpha_i, \quad \forall i = 0,1,\dots, n; \\
  s_0 \alpha_n & = \alpha_n + x \alpha_0; \\
  s_n \alpha_0 & = \alpha_0 + \frac{1}{x} \alpha_n; \\
  s_i \alpha_{i+1} = s_{i+1} \alpha_i & = \alpha_i + \alpha_{i+1}, \quad \forall i = 0,1, \dots, n-1;\\
  s_i \alpha_j & = \alpha_j, \quad \text{if $i \ne j$ and $i, j$ are not adjacent.}
\end{align*}
All of these representations are irreducible except when $x = 1$. 
If $x = 1$, the representation $V_1$ is nothing but the geometric representation in the sense of \cite[Chapter V, Section 4]{Bourbaki2002}, and it admits an $n$-dimensional simple quotient $V_1 / U$ where $U := \langle \alpha_0 + \alpha_1 + \cdots + \alpha_n \rangle$.
The quotient $V_1 / U$ is also a reflection representation.
Applying Theorems \ref{thm-main1} and \ref{thm-main2} yields uncountably many simple modules for the affine Weyl group $W$:
\[\Bigl\{\bigwedge^d V_x \Bigm| 1 \le d \le n, x \in \mathbb{C}^\times \setminus \{1\}\Bigr\} \cup \Bigl\{\bigwedge^d (V_1 / U) \Bigm| 1 \le d \le n-1 \Bigr\}.\]

\subsection{Outline of this paper}

The paper is organized as follows.
In Section \ref{sec-pre} we recollect the basic definitions and some preliminary results.
In Section \ref{sec-ext}, we recollect some basic results on exterior powers of reflection representations.
In Sections \ref{sec-main1} and \ref{sec-main2} we prove Theorems \ref{thm-main1} and \ref{thm-main2} respectively.

\section{Preliminaries} \label{sec-pre}

Throughout this paper, we work over a field $\mathbb{F}$ of characteristic 0.
We require $\operatorname{char} \mathbb{F} = 0$ only to ensure the exterior powers of an irreducible representation are semisimple (see Remark \ref{rmk-ss}).
In fact, the notions of reflections and reflection representations can be defined over fields of arbitrary characteristic.

For any positive integer $k$, we denote $[k] := \{1,2, \dots, k\}$.
For a fixed representation $\rho: W \to \operatorname{GL}(V)$ and an element $s \in W$, we also denote simply by $s$ the linear map $\rho(s) \in \operatorname{GL}(V)$ if there is no ambiguity.

\subsection{Reflections and reflection representations} \label{subsec-refl}

\begin{definition}[{\cite[Definition 2.1]{Hu23-ext-pow}}] \label{def-refl}
  Let $V$ be a finite-dimensional vector space  over $\mathbb{F}$.
  \begin{enumerate}
    \item A linear map $s:V \to V$ is called a  \emph{generalized reflection} (and \emph{reflection} for short) if $s$ is diagonalizable and $\operatorname{rank}(s - \operatorname{Id}_V) = 1$.
    \item Suppose $s$ is a reflection on $V$.
        The hyperplane $H_s : = \ker (s - \operatorname{Id}_V)$, which is fixed pointwise by $s$, is called the \emph{reflection hyperplane} of $s$.
        Let $\alpha_s$ be a nonzero vector in $\operatorname{Im}(s - \operatorname{Id}_V)$.
        Then, $s \cdot \alpha_s = \lambda_s \alpha_s$ for some $\lambda_s \in \mathbb{F} \setminus \{1\}$, and $\alpha_s$ is called a \emph{reflection vector} of $s$.
  \end{enumerate}
  Note that if $s$ is an invertible map, then $\lambda_s \ne 0$.
\end{definition}

The following lemma is immediate.

\begin{lemma} [{\cite[Lemma 2.2]{Hu23-ext-pow}}] \label{lem-refl}
  Let $s$ be a reflection on $V$ and $\alpha_s$ be a reflection vector.
  Then there exists a nonzero linear function $f : V \to \mathbb{F}$ such that $s \cdot v = v + f(v) \alpha_s$ for any $v \in V$.
\end{lemma}

The main object of our study, reflection representation, is defined as follows.

\begin{definition} \label{def-refl-rep}
  Let $W$ be a group endowed with a finite set of generators $S = \{ s_1, \dots, s_k\}$.
  A representation $(V,\rho)$ of $W$ over $\mathbb{F}$ is called a \emph{reflection representation of} $(W,S)$ if the linear map $\rho(s_i) \in \operatorname{GL}(V)$ is a reflection on $V$ for any $i \in [k]$.
\end{definition}

\subsection{Digraphs}

Digraphs will be helpful to investigate the structure of reflection representations.
In what follows we recall some relevant basic definitions.

By definition, a \emph{directed graph} (or \emph{digraph} for short) $G = (I,A)$ consists of a set $I$ of vertices and a set $A$ of arrows, where each arrow in $A$ is an ordered binary subset $(i,j)$ of $I$.
We also denote by $i \to j$ the arrow $(i,j)$.
For our purpose, we only consider finite digraphs without loops and multiple arrows, that is, (1) $I$ is a finite set, (2) there is no arrow of the form $i \to i$ and (3) each arrow $i \to j$ occurs at most once in $A$.

Suppose $i,j \in I$ are two vertices of a digraph $G$.
A \emph{walk} in $G$ from $i$ to $j$ is a sequence of vertices
\[i = i_0, \quad i_1, \quad \dots, \quad i_{l-1}, \quad i_l = j\]
such that $i_{m-1} \to i_m$ is an arrow in $A$ for each $m \in [l]$.

An \emph{undirected walk} in $G$ from $i$ to $j$ is an alternating sequence
\[i = i_0, \quad a_1, \quad i_1, \quad a_2, \quad i_2, \quad \dots, \quad i_{l-1}, \quad a_l, \quad i_l = j\]
of vertices $i_0, i_1, \dots, i_l \in I$ and arrows $a_1, \dots, a_l \in A$ such that either $a_m = i_{m-1} \to i_m$ or $a_m = i_m \to i_{m-1}$ for each $m \in [l]$.

A digraph $G$ is called \emph{weakly connected} if for any two vertices $i,j$ there exists an undirected walk from $i$ to $j$.
In other words, $G$ is weakly connected if the undirected graph obtained by forgetting the directions of all arrows in $A$ is connected.
Moreover, $G$ is called \emph{strongly connected} if for any two vertices $i,j$ there exists two walks, one from $i$ to $j$ and the other from $j$ to $i$.

Suppose $J \subset I$ is a subset of the vertices of $G$.
We define a digraph $G(J)$, called the \emph{sub-digraph spanned by $J$}, to be the digraph $(J, A(J))$ with the set $J$ of vertices, and the set $A(J) : = \{i \to j \mid i,j \in J, \text{ and } i \to j \text{ is an arrow in } A\}$ of arrows.

\begin{definition} \label{def-move}
  Let $G = (I,A)$ be a digraph and $J, J' \subseteq I$ be subsets of vertices.
  Suppose there exist vertices $i \in J$ and $j \in J'$ such that $i \to j$ is an arrow in $A$ and $J \setminus \{i\} = J' \setminus \{j\}$.
  Then we say $J'$ is obtained from $J$ by a \emph{move-forward}, and $J$ is obtained from $J'$ by a \emph{move-back}.
  We also say uniformly that $J$ or $J'$ is obtained from the other by a \emph{move}.

  Intuitively, we obtain $J'$ from $J$ by moving the vertex $i$ to the vertex $j$ along the arrow $i \to j$.
\end{definition}

The following lemma is essentially \cite[Lemma 4.3]{Hu23-ext-pow}.

\begin{lemma} \label{lem-move}
  Let $G=(I,A)$ be a weakly connected digraph.
  Let $J, J' \subseteq I$ be two subsets with the same cardinality.
  Then $J'$ can be obtained from $J$ by finite steps of moves.
\end{lemma}

\begin{proof}
  Forgetting the directions of arrows in $A$, this lemma follows from \cite[Lemma 4.3]{Hu23-ext-pow}.
\end{proof}

Digraphs and reflection representations are related via the following definition.

\begin{definition} \label{def-assoc-digr}
  Let $W$ be a group endowed with a finite set of generators $S = \{ s_1, \dots, s_k\}$, and $(V,\rho)$ be a reflection representation of $(W,S)$.
  For each $i \in [k]$, let $\alpha_i$ be an arbitrarily chosen reflection vector of $s_i$.
  For any subset $I \subseteq [k]$, we define the \emph{associated digraph} $G_I$ to be a digraph $(I,A)$ where $I$ is the set of vertices and
  \[A := \{i \to j \mid i, j \in I, s_j \cdot \alpha_i \ne \alpha_i\}\]
  is the set of arrows.
  We also denote simply by $G$ the associated digraph $G_{[k]}$.

  Clearly, for subsets $J \subseteq I \subseteq [k]$, the digraph $G_J$ is the sub-digraph $G_I(J)$ of $G_I$ spanned by $J$.
\end{definition}

Immediately we have the following fact about the associated digraph.

\begin{lemma} \label{lem-assoc-digr}
  If $i \to j$ is an arrow in the digraph $G_I$, then $\alpha_j$ belongs to the subrepresentation generated by $\alpha_i$.
\end{lemma}

\begin{proof}
  By definition, we have $s_j \cdot \alpha_i \ne \alpha_i$.
  In view of Lemma \ref{lem-refl}, the vector $\alpha_i - s_j \cdot \alpha_i$ is a nonzero multiple of $\alpha_j$.
  But this vector lies in the subrepresentation generated by $\alpha_i$.
\end{proof}

\subsection{Some numerical lemmas}

We will need the following lemmas.

\begin{lemma} \label{lem-binom1}
  Let $n_1, n_2, d_1, d_2$ be positive integers and $1 \le d_\iota \le n_\iota - 1$ for $\iota = 1,2$.
  Suppose $\frac{d_1}{n_1} = \frac{d_2}{n_2}$ and $\binom{n_1}{d_1} = \binom{n_2}{d_2}$.
  Then $n_1 = n_2$ and $d_1 = d_2$.
\end{lemma}

\begin{proof}
  Without loss of generality, we may assume $n_1 \le n_2$ and $d_\iota \le \frac{n_\iota}{2}$ for $\iota = 1, 2$.
  Suppose  $n_1 < n_2$.
  Then $d_1 < d_2$.
  Then we have $\binom{n_2}{d_2} > \binom{n_2}{d_1} > \binom{n_1}{d_1}$ which is a contradiction.
  Therefore, $n_1 = n_2$ and hence $d_1  = d_2$.
\end{proof}

\begin{lemma} \label{lem-binom2}
  Let $n_1, n_2, d_1, d_2$ be positive integers and $1 \le d_\iota \le n_\iota - 1$ for $\iota = 1,2$.
  Suppose
  \begin{equation}\label{eq-binom2-1}
    \binom{n_1 - 1}{d_1} = \binom{n_2 - 1}{d_2}
  \end{equation}
  and
  \begin{equation}\label{eq-binom2-2}
    \binom{n_1 - 1}{d_1 - 1} = \binom{n_2 - 1}{d_2 - 1}.
  \end{equation}
  Then $n_1 = n_2$ and $d_1 = d_2$.
\end{lemma}

\begin{proof}
  By direct computations, for $\iota = 1, 2$ we have
  \begin{align*}
     \binom{n_\iota - 1}{d_\iota} - \binom{n_\iota - 1}{d_\iota - 1} & = \frac{(n_\iota - 1)!}{d_\iota ! (n_\iota - d_\iota - 1)!} - \frac{(n_\iota - 1) !}{(d_\iota - 1 ) ! (n_\iota - d_\iota) !} \\
     & = \frac{(n_\iota - 1)! }{d_\iota ! (n_\iota - d_\iota)!} (n_\iota - 2 d_\iota) \\
     & = \binom{n_\iota}{d_\iota} \bigl(1 - \frac{2 d_\iota}{n_\iota} \bigr).
  \end{align*}
  By Equations \eqref{eq-binom2-1} and \eqref{eq-binom2-2} we then have
  \begin{equation}\label{eq-binom2-3}
    \binom{n_1}{d_1} \bigl(1 - \frac{2 d_1}{n_1}\bigr) = \binom{n_2}{d_2} \bigl(1 - \frac{2 d_2}{n_2}\bigr).
  \end{equation}
  Adding the Equations \eqref{eq-binom2-1} and \eqref{eq-binom2-2} together yields
  \begin{equation}\label{eq-binom2-5}
    \binom{n_1}{d_1} = \binom{n_2}{d_2}.
  \end{equation}
  We combine Equations \eqref{eq-binom2-3} and \eqref{eq-binom2-5}, then we obtain
  \begin{equation*}
    \frac{d_1}{n_1} = \frac{d_2}{n_2}.
  \end{equation*}
  By Lemma \ref{lem-binom1} we have $n_1 = n_2$ and $d_1 = d_2$.
\end{proof}

\section{Exterior powers of reflection representations} \label{sec-ext}

In this section we recollect some first results about exterior powers.
Let $W$ be a group endowed with a set of generators $S = \{ s_1, \dots, s_k\}$ as before.
Suppose $(V,\rho)$ is an $n$-dimensional representation of $W$.
The action $\bigwedge^d \rho$ of $W$ on the $d$th exterior power $\bigwedge^d V$ ($0 \le d \le n$) is given by
\[w \cdot (v_1 \wedge \dots \wedge v_d) = (w \cdot v_1) \wedge \dots \wedge (w \cdot v_d), \quad \forall w \in W,  v_1, \dots, v_d \in V.\]
In particular, $\bigwedge^0 V$ is the one-dimensional $W$-module with trivial action, and $\bigwedge^n V$ carries the one-dimensional representation $\det \circ \rho$.
If $\{v_1, \dots, v_n\}$ is a basis of $V$, then 
\[\{v_{i_1} \wedge \dots \wedge v_{i_d} \mid 1 \le i_1 < \dots < i_d \le n\}\]
is a basis of $\bigwedge^d V$.
In particular, $\dim \bigwedge^d V = \binom{n}{d}$.
For more details about exterior powers, see, for example, \cite{FH91}.

Suppose further that $(V,\rho)$ is a reflection representation and $\alpha_i$ is a chosen reflection vector of $s_i$ with eigenvalue $\lambda_i$ ($\ne 1$) for each $i \in [k]$ (see Definitions \ref{def-refl} and \ref{def-refl-rep}).
We also denote by $H_i$ the reflection hyperplane of $s_i$.
For each $i \in [k]$ and $0 \le d\le n$, we define
\[V_{d,i}^+ := \Bigl\{v \in \bigwedge^d V \Bigm| s_i \cdot v = v\Bigr\}, \quad V_{d,i}^- := \Bigl\{v \in \bigwedge^d V \Bigm| s_i \cdot v = \lambda_i v\Bigr\}\]
to be the eigen-subspaces of $s_i$ in $\bigwedge^d V$, for the eigenvalues $1$ and $\lambda_i$, respectively.

Retain the notations $W, V, s_i, \alpha_i$, etc.

\begin{lemma} [{\cite[Lemma 3.2 and Corollary 3.3]{Hu23-ext-pow}}] \label{lem-eigenspace}
  Let $i \in [k]$ and $0 \le d \le n$.
  \begin{enumerate}
    \item \label{lem-eigenspace-1} We have $V^+_{d,i} = \bigwedge^d H_i$ and $\dim V^+_{d,i} = \binom{n-1}{d}$.
        Here we regard $\binom{n-1}{n} = 0$ if $d = n$.
    \item \label{lem-eigenspace-2} Extend the reflection vector $\alpha_i$ arbitrarily to a basis of $V$, say, $\alpha_i, v_2, \dots, v_n$.
        Then, $V^-_{d,i}$ has a basis
        \[\{\alpha_i \wedge v_{i_1} \wedge \dots \wedge v_{i_{d-1}} \mid 2 \le i_1 < \dots < i_{d-1} \le n\}.\]
        In particular, $\dim V^-_{d,i} = \binom{n-1}{d-1}$.
        Here we regard $\binom{n-1}{-1} = 0$ if $d = 0$.
    \item \label{lem-eigenspace-3} As a vector space, $\bigwedge^d V = V^+_{d, i} \bigoplus V^-_{d, i}$.
        In particular, the only possible eigenvalues of $s_i$ on $\bigwedge^d V$ are $1$ and $\lambda_i$.
  \end{enumerate}
\end{lemma}

\begin{lemma} [{\cite[Proposition 3.5]{Hu23-ext-pow}}] \label{lem-intersect-eigenspace}
  Suppose the reflection vectors $\alpha_1, \dots, \alpha_m$ ($m \le k$) are linearly independent. We extend these vectors to a basis of $V$, say,
  \[\{\alpha_1, \dots, \alpha_m, v_{m+1}, \dots, v_n\}.\]
  \begin{enumerate}
    \item If $0 \le d < m$, then $\bigcap_{1 \le i \le m} V^-_{d,i} = 0$.
    \item If $m \le d \le n$, then $\bigcap_{1 \le i \le m} V^-_{d,i}$ has a basis
        \[\{\alpha_1 \wedge \dots \wedge \alpha_m \wedge v_{i_{m+1}} \wedge \dots \wedge v_{i_d} \mid m+1 \le i_{m+1} < \dots < i_d \le n \}.\]
        In particular, if $d = m$, then $\bigcap_{1 \le i \le m} V^-_{d,i}$ is one-dimensional with a basis vector $\alpha_1 \wedge \dots \wedge \alpha_m$.
  \end{enumerate}
\end{lemma}

\begin{lemma} \label{lem-extra-eigen}
  Suppose $m \le d$, $m \le k-1$, and the reflection vectors $\alpha_1, \dots, \alpha_m$ are linearly independent.
  Suppose $\alpha_{m+1}$ is a linear combination of $\alpha_1, \dots, \alpha_m$.
  Then $\bigcap_{1 \le i \le m+1} V_{d,i}^- = \bigcap_{1 \le i \le m} V_{d,i}^- \ne 0$ (that is, $s_{m+1} \cdot v = \lambda_{m+1} v$ for any $v \in \bigcap_{1 \le i \le m} V_{d,i}^-$).
\end{lemma}

\begin{proof}
  The fact that $\bigcap_{1 \le i \le m} V_{d,i}^- \ne 0$ follows from Lemma \ref{lem-intersect-eigenspace}.
  Moreover, the subspace $\bigcap_{1 \le i \le m} V_{d,i}^-$ admits a basis of the form
  \begin{equation}\label{eq-lem-extra-eigen}
    \{\alpha_1 \wedge \dots \wedge \alpha_m \wedge v_{i_{m+1}} \wedge \dots \wedge v_{i_d} \mid  m+1 \le i_{m+1} < \dots < i_d \le n \}
  \end{equation}
  where $\{\alpha_1, \dots, \alpha_m, v_{m+1}, \dots, v_n\}$ is a basis of $V$.

  Suppose $\alpha_{m+1} = c_1 \alpha_1 + \dots + c_m \alpha_m$, $c_i \in \mathbb{F}$.
  Without loss of generality, we may assume further $c_1 \ne 0$.
  Then for any basis vector in \eqref{eq-lem-extra-eigen} we have
  \begin{align*}
     & \mathrel{\phantom{=}} \alpha_1 \wedge \dots \wedge \alpha_m \wedge v_{i_{m+1}} \wedge \dots \wedge v_{i_d} \\
     & = c_1 ^{-1} (c_1 \alpha_1) \wedge \alpha_2 \wedge \dots \wedge \alpha_m \wedge v_{i_{m+1}} \wedge \dots \wedge v_{i_d} \\
     & = c_1^{-1} (c_1 \alpha_1 + \dots + c_m \alpha_m) \wedge \alpha_2 \wedge \dots \wedge \alpha_m \wedge v_{i_{m+1}} \wedge \dots \wedge v_{i_d} \\
     & = c_1^{-1} \alpha_{m+1} \wedge \alpha_2 \wedge \dots \wedge \alpha_m \wedge v_{i_{m+1}} \wedge \dots \wedge v_{i_d}.
  \end{align*}
  Note that this is a nonzero vector, and that $\{\alpha_{m+1}, \alpha_2, \dots, \alpha_m, v_{m+1}, \dots, v_n\}$ is also a basis of $V$.
  By Lemma \ref{lem-intersect-eigenspace} again, we have
  \[\alpha_{m+1} \wedge \alpha_2 \wedge \dots \wedge \alpha_m \wedge v_{i_{m+1}} \wedge \dots \wedge v_{i_d} \in V_{d,m+1}^-.\]
  Therefore, $\bigcap_{1 \le i \le m} V_{d,i}^- \subseteq V_{d,m+1}^-$, and thus $\bigcap_{1 \le i \le m+1} V_{d,i}^- = \bigcap_{1 \le i \le m} V_{d,i}^-$.
\end{proof}

\begin{lemma} [{\cite[Proposition 3.6]{Hu23-ext-pow}}] \label{lem-ext-num}
  If $0 \le d, d' \le n$ are integers and $\bigwedge^d V \simeq \bigwedge^{d'} V$ as $W$-modules, then $d = d'$.
\end{lemma}

\begin{remark} \label{rmk-ext-num}
  Lemma \ref{lem-ext-num} holds for any representation on which some element $s \in W$ acts by a reflection, not necessary a reflection representation.
  See \cite[Proposition 3.6]{Hu23-ext-pow} for details.
\end{remark}

\begin{lemma} [{\cite[Corollary 3.8]{Hu23-ext-pow}}] \label{lem-ss}
  If the representation $(V,\rho)$ is irreducible, then the $W$-module $\bigwedge^d V$ is semisimple for any $d = 0, 1, \dots, n$.
\end{lemma}

\begin{remark} \label{rmk-ss}
  Recall that $\operatorname{char} \mathbb{F}$ is assumed to be  $0$.
  This is used in the proof of Lemma \ref{lem-ss}.
  See \cite[Lemma 3.7 and Corollary 3.8]{Hu23-ext-pow} for details.
\end{remark}

\section{Proof of Theorem \ref{thm-main1}} \label{sec-main1}

In this section we give the proof of Theorem \ref{thm-main1}.

Recall that $W$ is a group endowed with a set of generators $S = \{ s_1, \dots, s_k\}$, and $(V,\rho)$ is an $n$-dimensional irreducible reflection representation of $(W,S)$ over a field $\mathbb{F}$ of characteristic 0.
We denote by  $\alpha_i$ the chosen reflection vector of $s_i$ as before, and by $\lambda_i$ ($\ne 1$) the corresponding eigenvalue, for each $i \in [k]$.

By Lemma \ref{lem-ext-num}, the $W$-modules $\{\bigwedge^d V \mid 0 \le d \le n\}$ are pairwise non-isomorphic.
Therefore, to prove Theorem \ref{thm-main1}, it suffices to show that $\bigwedge^d V$ is a simple $W$-module for each $d$.
But we have seen in Lemma \ref{lem-ss} that $\bigwedge^d V$ is semisimple,
so the problem reduces to proving
\begin{equation} \label{eq-main1}
  \text{any endomorphism of $\bigwedge^d V$ is a scalar multiplication.}
\end{equation}

Recall in Definition \ref{def-assoc-digr} that  a digraph $G_I$ is associated with the reflection representation $(V,\rho)$ and an arbitrary subset $I \subseteq [k]$.
We have the following lemma.

\begin{lemma} \label{lem-GI}
  There exists a subset $I \subseteq [k]$ such that
  \begin{enumerate}
    \item \label{lem-GI-1} the digraph $G_I$ is weakly connected, and
    \item \label{lem-GI-2} $\{\alpha_i \mid i \in I\}$ is a basis of $V$.
  \end{enumerate}
\end{lemma}

\begin{proof}
  Suppose we have found a subset $J \subseteq [k]$ such that
  \begin{enumerate}
    \item [(a)] the digraph $G_J$ is weakly connected, and
    \item [(b)] $\{\alpha_i \mid i \in J\}$ is linearly independent.
  \end{enumerate}
  For example, any singleton $\{j\} \subseteq [k]$ is such a subset.

  If $\lvert J \rvert = n$ (the dimension of $V$), then we are done.
  Otherwise, suppose $\lvert J \rvert < n$.
  Let $V_J := \bigoplus_{i \in J} \mathbb{F} \alpha_i$, which is a proper subspace of $V$.
  Since $V$ is a simple $W$-module, there exists $j \in J$ and $i_0 \in [k]$ such that  $s_{i_0} \cdot \alpha_j \notin V_J$.
  By Lemma \ref{lem-refl}, $s_{i_0} \cdot \alpha_j$ is of the form
  \[s_{i_0} \cdot \alpha_j = \alpha_j + x \alpha_{i_0}\]
  for some $x \in \mathbb{F}$.
  Then we must have $x \ne 0$ and ${i_0} \notin J$, otherwise $s_{i_0} \cdot \alpha_j$ would belong to $V_J$.
  Now let $J' = J \sqcup \{{i_0}\}$.
  Then the associated digraph $G_{J'}$ is also weakly connected since we have an arrow $j \to {i_0}$.
  Moreover, the set of vectors $\{\alpha_i \mid i \in J'\}$ is linearly independent since $\alpha_{i_0} \notin \bigoplus_{i \in J} \mathbb{F} \alpha_i$.
  Therefore, the subset $J'$ satisfies the conditions (a) and (b).
  Moreover, we have $\lvert J' \rvert = \lvert J \rvert + 1$.
  By induction on cardinality, there exists a subset $I \subseteq [k]$ satisfying \eqref{lem-GI-1} and \eqref{lem-GI-2}.
\end{proof}

The following corollary of Lemma \ref{lem-GI} will be used in Section \ref{sec-main2}.

\begin{corollary} [See also {\cite[Claim 5.2]{Hu23-ext-pow}}] \label{cor-span-refl}
  Suppose $(V,\rho)$ is an $n$-dimensional irreducible reflection representation of $(W,S)$ with reflection vectors $\{\alpha_i \mid i \in [k]\}$.
  Then the space $V$ is spanned by $\{\alpha_i \mid i \in [k]\}$, that is, $V = \sum_{i \in [k]} \mathbb{F} \alpha_i$.
  In particular, $n \le k$.
\end{corollary}

Let $I$ be obtained as in Lemma \ref{lem-GI}.
Without loss of generality, we may assume $I = [n]$, the first $n$ indices of $[k]$ (note that we have $n \le k$ by Lemma \ref{lem-GI}).
The vectors $\{\alpha_i \mid i \in I\}$ form a basis of $V$. 
For each fixed $d$ with $0 \le d \le n$, the set of vectors
\[\{\alpha_{i_1} \wedge \dots \wedge \alpha_{i_d} \mid 1 \le i_1 < \dots < i_d \le n\}\]
is a basis of $\bigwedge^d V$.

For any set of distinct indices $1 \le i_1, \dots, i_d \le n$, by Lemma \ref{lem-intersect-eigenspace}, the intersection $\bigcap_{1 \le j \le d} V_{d, i_j}^-$ of the $d$ eigen-subspaces is one-dimensional,
\[\bigcap_{1 \le j \le d} V_{d, i_j}^- = \mathbb{F} \alpha_{i_1} \wedge \dots \wedge \alpha_{i_d}.\]
Suppose now $\varphi \in \operatorname{End}_W (\bigwedge^d V)$ is an endomorphism.
Then $\varphi$ preserves the subspace $\bigcap_{1 \le j \le d} V_{d, i_j}^-$.
Therefore,
\[\text{$\varphi (\alpha_{i_1} \wedge \dots \wedge \alpha_{i_d}) = \gamma_{i_1, \dots, i_d} \alpha_{i_1} \wedge \dots \wedge \alpha_{i_d}$ for some $\gamma_{i_1, \dots, i_d} \in \mathbb{F}$.}\]
Notice that $\alpha_{i_{\sigma(1)}} \wedge \dots \wedge \alpha_{i_{\sigma(d)}} = \operatorname{sign}(\sigma) \alpha_{i_1} \wedge \dots \wedge \alpha_{i_d}$ for any permutation $\sigma \in \mathfrak{S}_d$, and hence that $\gamma_{i_1, \dots, i_d}$ depends only on the set $\{i_1, \dots, i_d\}$, not on the order of the indices.
To prove the statement \eqref{eq-main1}, it suffices to show that the coefficients $\gamma_{i_1, \dots, i_d}$ are independent of the choice of the indices $\{i_1, \dots, i_d\}$.
The following result is essentially the same as \cite[Claim 5.5]{Hu23-ext-pow}.

\begin{lemma}
  Let $J = \{i_1, \dots, i_d\}$, $J' = \{j_1, \dots, j_d\}$ be two subsets of $I$, both consisting of $d$ elements.
  Suppose $J'$ can be obtained from $J$ by a move (see Definition \ref{def-move}) in the digraph $G_I$.
  Then $\gamma_{i_1, \dots, i_d} = \gamma_{j_1, \dots, j_d}$.
\end{lemma}

\begin{proof}
  Without loss of generality, we may assume that $d \le n-1$, $J=\{1, \dots, d\}$, $J' = \{1,2, \dots, d-1, d+1\}$, and $d \to d+1$ is an arrow in $G_I$.
  Then $s_{d+1} \cdot \alpha_d \ne \alpha_d$.

  For $i = 1, \dots, d$, by Lemma \ref{lem-refl} we assume that
  \[s_{d+1} \cdot \alpha_i = \alpha_i + c_i \alpha_{d+1}, \quad c_i \in \mathbb{F}.\]
  Then $c_d \ne 0$.
  We have
  \begin{align*}
    & \mathrel{\phantom{=}} s_{d+1} \cdot (\alpha_1 \wedge \dots \wedge \alpha_d) \\
     & = (\alpha_1 + c_1 \alpha_{d+1}) \wedge \dots \wedge (\alpha_d + c_d \alpha_{d+1})\\
     & = \alpha_1 \wedge \dots \wedge \alpha_d   + \sum_{1 \le i \le n} (-1)^{d-i} c_i \cdot \alpha_1 \wedge \dots \wedge \widehat{\alpha}_i \wedge \dots \wedge \alpha_d \wedge \alpha_{d+1}.
  \end{align*}
  Hence,
  \begin{align*}
    & \mathrel{\phantom{=}} \varphi \bigl(s_{d+1} \cdot (\alpha_1 \wedge \dots \wedge \alpha_d)\bigr) \\
    & = \varphi \bigl(\alpha_1 \wedge \dots \wedge \alpha_d  + \sum_{i=1}^{d} (-1)^{d-i} c_i \cdot \alpha_1 \wedge \dots \wedge \widehat{\alpha}_i \wedge \dots \wedge  \alpha_{d+1}\bigr) \\
     & = \gamma_{1,\dots, d} \cdot \alpha_1 \wedge \dots \wedge \alpha_d  + \sum_{i=1}^{d} (-1)^{d-i} c_i \gamma_{1,\dots,\widehat{i},\dots, d+1} \cdot \alpha_1 \wedge \dots \wedge \widehat{\alpha}_i \wedge \dots \wedge  \alpha_{d+1}.
  \end{align*}
  This also equals
  \begin{align*}
     & \mathrel{\phantom{=}} s_{d+1} \cdot \varphi(\alpha_1 \wedge \dots \wedge \alpha_d) \\
     & = \gamma_{1,\dots, d} s_{d+1} \cdot (\alpha_1 \wedge \dots \wedge \alpha_d) \\
     & = \gamma_{1,\dots, d} \cdot \alpha_1 \wedge \dots \wedge \alpha_d + \sum_{i=1}^{d} (-1)^{d-i} c_i \gamma_{1,\dots, d} \cdot \alpha_1 \wedge \dots \wedge \widehat{\alpha}_i \wedge \dots \wedge  \alpha_{d+1}.
  \end{align*}
  Note that $c_d \ne 0$, and that the vectors involved in the summations above are linearly independent.
  Thus, we have the desired equality $\gamma_{1, \dots, d} = \gamma_{1, \dots, d-1, d+1}$ by comparing the coefficients of $\alpha_1 \wedge \dots \wedge \alpha_{d-1} \wedge \alpha_{d+1}$.
\end{proof}

In general, for two subsets $J$ and $J'$ of $I$, if both of them consist of $d$ elements, then, since $G_I$ is weakly connected, one can be obtained from the other by finite steps of moves by Lemma \ref{lem-move}.
Therefore, the coefficients $\gamma_{i_1, \dots, i_d}$ are constant among all choices of the distinct indices $1 \le i_1, \dots, i_d \le n$.

The proof of Theorem \ref{thm-main1} is completed.

\begin{remark}
  We cannot expect the digraph $G_I$ in Lemma \ref{lem-GI}  to be strongly connected.
  For example, let $S = \{s_1, s_2, s_3\}$ consist of 3 elements, and $V = \mathbb{F} \alpha_1 \oplus \mathbb{F} \alpha_2$ be a two-dimensional vector space.
  Define three reflections on $V$ by
  \begin{alignat*}{2}
    s_1 \cdot \alpha_1 & = - \alpha_1,   & s_1 \cdot \alpha_2 & = \alpha_2, \\
    s_2 \cdot \alpha_1 & = \alpha_1 + 2 \alpha_2, \quad \quad & s_2 \cdot \alpha_2 & = -\alpha_2, \\
    s_3 \cdot \alpha_1 & = \alpha_1, & s_3 \cdot \alpha_2 & = -2\alpha_1 - \alpha_2.
  \end{alignat*}
  Then the corresponding reflection vectors are $\alpha_1$, $\alpha_2$, and $\alpha_3:= - \alpha_1 - \alpha_2$, respectively.
  The associated digraph $G_{[3]}$ is as follows:
  
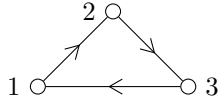
\begin{equation*}
    \begin{tikzpicture}
      \node [circle, draw, inner sep=2pt, label=left:$1$] (s1) at (0,0) {};
      \node [circle, draw, inner sep=2pt, label=left:$2$] (s2) at (1,1) {};
      \node [circle, draw, inner sep=2pt, label=right:$3$] (s3) at (2,0) {};
      \draw  (s1) -- (s2) -- (s3) -- (s1);
      \draw ($(0.55,0.55) + (195:0.2)$) -- (0.55,0.55) -- ($(0.55,0.55) + (255:0.2)$);
      \draw ($(1.55, 0.45) + (105:0.2)$) -- (1.55, 0.45) -- ($(1.55, 0.45) + (165:0.2)$);
      \draw ($(0.95,0) + (30:0.2)$) -- (0.95,0) -- ($(0.95,0) + (330:0.2)$);
    \end{tikzpicture}
  \end{equation*}
  In this digraph, each sub-digraph spanned by two vertices is not strongly connected.
\end{remark}

\section{Proof of Theorem \ref{thm-main2}} \label{sec-main2}

This section is devoted to proving Theorem \ref{thm-main2}.

Recall that $W$ is a group endowed with a set of generators $S = \{ s_1, \dots, s_k\}$, and $(V_\iota, \rho_\iota)$, $\iota = 1, 2$, are two irreducible reflection representations.
We use the following notations.
\begin{alignat*}{2}
  & \text{$n_\iota$ ($\iota = 1, 2$) :} \quad & & \text{$\dim V_\iota$} \\
  & \text{$\alpha_i$ ($i \in [k]$) :} \quad & & \text{the chosen reflection vector of $s_i$ in $V_1$} \\
  & \text{$\lambda_i$ ($\ne 1$) :} \quad & & \text{the corresponding eigenvalue, $s_i \cdot \alpha_i = \lambda_i \alpha_i$} \\
  & \text{$\beta_i$ ($i \in [k]$) :} \quad & & \text{the chosen reflection vector of $s_i$ in $V_2$} \\
  & \text{$\mu_i$ ($\ne 1$) :} \quad & & \text{the corresponding eigenvalue, $s_i \cdot \beta_i = \mu_i \beta_i$}
\end{alignat*}

Suppose
\[\psi: \bigwedge^{d_1} V_1 \xrightarrow{\sim} \bigwedge^{d_2} V_2\]
is an isomorphism of $W$-modules, where $d_1, d_2$ are certain integers satisfying $1 \le d_\iota \le n_\iota -1$ ($\iota = 1,2$).
As in Section \ref{sec-ext}, for each $i \in [k]$ we denote by
\begin{alignat*} {2}
  V_{1, d_1, i}^+ & := \Bigl\{ v \in \bigwedge^{d_1} V_1 \Bigm | s_i \cdot v = v \Bigr\}, \quad & V_{1, d_1, i}^- & := \Bigl\{ v \in \bigwedge^{d_1} V_1 \Bigm | s_i \cdot v = \lambda_i v \Bigr\}, \\
  V_{2, d_2, i}^+ & := \Bigl\{ v \in \bigwedge^{d_2} V_2 \Bigm | s_i \cdot v = v \Bigr\}, \quad & V_{2, d_2, i}^- & := \Bigl\{ v \in \bigwedge^{d_2} V_2 \Bigm | s_i \cdot v = \mu_i v \Bigr\}
\end{alignat*}
the eigen-subspaces of $s_i$.

Before giving the rigorous proof, let us talk a little more about Theorem \ref{thm-main2} informally.
A priori, an isomorphism $f : V_1 \to V_2$ of reflection representations gives an isomorphism $\bigwedge^d f : \bigwedge^d V_1 \to \bigwedge^d V_2$ via
\[\Bigl( \bigwedge^d f \Bigr) (v_1 \wedge \dots \wedge v_d) = f (v_1) \wedge \dots \wedge f (v_d), \quad \forall v_1, \dots, v_d \in V_1. \]
It is not difficult to see that $f(\alpha_i) = z_i \beta_i$ for some $z_i \in \mathbb{F}^\times$.
Then we have
\[\Bigl( \bigwedge^d f \Bigr) (\alpha_{i_1} \wedge \dots \wedge \alpha_{i_d}) = z_{i_1} \cdots z_{i_d} \beta_{i_1} \wedge \dots \wedge \beta_{i_d} \text{ for any } i_1, \dots, i_d \in [k].\]
Conversely suppose in Theorem \ref{thm-main2} that $d = d_1 = d_2$, and that the isomorphism $\psi: \bigwedge^d V_1 \to \bigwedge^d V_2$ is given by an isomorphism $f : V_1 \to V_2$.
Suppose further that we are able to show for any indices $i_1, \dots, i_d \in [k]$ that
\[\psi(\alpha_{i_1} \wedge \dots \wedge \alpha_{i_d}) = \zeta_{i_1, \dots, i_d} \beta_{i_1} \wedge \dots \wedge \beta_{i_d} \text{ for some } \zeta_{i_1, \dots, i_d} \in \mathbb{F}^\times.\]
(This is indeed the case, see Subsection \ref{subsec-coinc-digraphs}.)
Since the map $f$ is of the form $f(\alpha_i) = z_i \beta_i$, we have $\zeta_{i_1, \dots, i_d} = z_{i_1} \cdots z_{i_d}$, and
\[\frac{z_i}{z_j} = \frac{\zeta_{i, i_2, \dots, i_d}}{\zeta_{j, i_2, \dots, i_d}} \text{ for any suitable indices } i, j, i_2, \dots, i_d \in [k]. \]
This indicates that
\begin{equation}\label{eq-informal}
  \begin{split}
    &\text{the ratio $\frac{\zeta_{i, i_2, \dots, i_d}}{\zeta_{j, i_2, \dots, i_d}}$ only depends on $i$ and $j$,}   \\
    &\text{but independent of the indices $i_2, \dots, i_d$.}
  \end{split}
\end{equation}
We would be close to find the desired isomorphism $f$ if we can prove \eqref{eq-informal} (this is essentially Lemma \ref{lem-same-digraph-2}).

We divide the proof of Theorem \ref{thm-main2} into the following five steps, presented in Subsections \ref{subsec-pre-num} to \ref{subsec-isom} respectively:
\begin{enumerate}
  \item [Step 1.] Show that $d_1 = d_2$, $n_1 = n_2$, and $\lambda_i = \mu_i$ for each $i \in [k]$.
  \item [Step 2.] Show that the linear independence of a set of reflection vectors in $V_1$ is equivalent to that in $V_2$.
  \item [Step 3.] Show that the two reflection representations have the same associated digraphs.
  \item [Step 4.] Define a linear isomorphism $f: V_1 \to V_2$ of vector spaces.
  \item [Step 5.] Show that $f$ is an isomorphism of $W$-modules.
\end{enumerate}

\subsection{A preliminary numerical result} \label{subsec-pre-num}

\begin{proposition} \label{prop-num}
  $d_1 = d_2$, $n_1 = n_2$. Moreover, $\lambda_i = \mu_i$ for each $i \in [k]$.
\end{proposition}

\begin{proof}
  Note that the element $s_1 \in S$  acts by reflections on both $V_1$ and $V_2$.
  Since $\bigwedge^{d_1} V_1 \simeq \bigwedge^{d_2} V_2$ as $W$-modules, we have
  \[\dim V_{1, d_1, 1}^+ = \dim V_{2, d_2, 1}^+, \quad \dim V_{1, d_1,1}^- = \dim V_{2, d_2, 1}^-.\]
  Then we have by Lemma \ref{lem-eigenspace}\eqref{lem-eigenspace-1}\eqref{lem-eigenspace-2}
  \begin{equation}\label{eq-prop-num}
    \binom{n_1-1}{d_1} = \binom{n_2-1}{d_2}, \quad \binom{n_1-1}{d_1-1} = \binom{n_2-1}{d_2-1}.
  \end{equation}
  Notice that $1 \le d_\iota \le n_\iota - 1$ for $\iota = 1,2$.
  By Lemma \ref{lem-binom2}, Equations \eqref{eq-prop-num} imply $d_1 = d_2$, $n_1 = n_2$.

  By Lemma \ref{lem-eigenspace}\eqref{lem-eigenspace-3}, we have $\bigwedge^{d_1} V_1 = V_{1,d_1,i}^+ \bigoplus V_{1,d_1,i}^-$ for each $i \in [k]$, and the only possible eigenvalues of $s_i$ on $\bigwedge^{d_1} V_1$ are $1$ and $\lambda_i$.
  But $\dim V_{1,d_1,i}^- = \binom{n_1 - 1}{d_1 -1} \ne 0$ since $1 \le d_1 \le n_1 -1$.
  Thus $\lambda_i$ is indeed an eigenvalue.
  Similarly, the only eigenvalues of $s_i$ on $\bigwedge^{d_2} V_2$ are $1$ and $\mu_i$.
  Thus we must have $\lambda_i = \mu_i$.
\end{proof}

\begin{remark}
  From the proof of Proposition \ref{prop-num}, we see that the results hold for two representations $(V_1,\rho_1)$, $(V_2,\rho_2)$ on which $\rho_1(s)$ and $\rho_2(s)$ are both reflections for some element $s \in W$, not necessary to be reflection representations.
\end{remark}

In view of Proposition \ref{prop-num}, we denote
\[d := d_1 = d_2 \quad \text{and} \quad n := \dim V_1 = \dim V_2\]
from now on.
Note that we have $1 \le d \le n-1$ by assumption.
However, Theorem \ref{thm-main2} for the case $d = 1$ is trivial.
Thus we may assume $2 \le d \le n-1$.

\subsection{Preliminary results on linear independence of reflection vectors}

This subsection aims to prove  Propositions \ref{prop-corr-basis} and \ref{prop-tran-lin-ind}, which transfer linear independence property of reflection vectors in $V_1$ to those with the same indices in $V_2$.
Recall that $k = \lvert S \rvert$ is the number of chosen generators of the group $W$, and $\alpha_i$, $\beta_i$ ($i \in [k]$) are the reflection vectors of the generator $s_i$ in the space $V_1$ and $V_2$ respectively.
By Corollary \ref{cor-span-refl}, the $n$-dimensional vector space $V_1$ is spanned by $\alpha_1, \dots, \alpha_k$, and we have $n \le k$.

\begin{lemma} \label{lem-com-eigen-decomp}
  Suppose $\{\alpha_1, \dots, \alpha_n\}$ is a basis of $V_1$.
  Then we have decompositions of vector spaces
  \[\bigwedge^d V_1 = \bigoplus_{1 \le i_1 < \dots < i_d \le n} \Bigl( \bigcap_{1 \le j \le d} V_{1,d,i_j}^- \Bigr)\]
  and
  \[\bigwedge^d V_2 = \bigoplus_{1 \le i_1 < \dots < i_d \le n} \Bigl( \bigcap_{1 \le j \le d} V_{2,d,i_j}^- \Bigr).\]
\end{lemma}

\begin{proof}
  The vector space $\bigwedge^d V_1$ has a basis
  \[\{\alpha_{i_1} \wedge \dots \wedge \alpha_{i_d} \mid 1 \le i_1 < \dots < i_d \le n \}.\]
  Note that by Lemma \ref{lem-intersect-eigenspace} the vector $\alpha_{i_1} \wedge \dots \wedge \alpha_{i_d}$ is a basis vector of the one-dimensional space $\bigcap_{1 \le j \le d} V_{1,d,i_j}^-$, which is the intersection of eigen-subspaces of $s_{i_j}$'s in $\bigwedge^d V_1$.
  Therefore, we have a decomposition of vector space
  \begin{align*}
    \bigwedge^d V_1 & = \bigoplus_{1 \le i_1 < \dots < i_d \le n} \mathbb{F} \langle \alpha_{i_1} \wedge \dots \wedge \alpha_{i_d} \rangle \\
     & = \bigoplus_{1 \le i_1 < \dots < i_d \le n} \Bigl( \bigcap_{1 \le j \le d} V_{1,d,i_j}^- \Bigr).
  \end{align*}
  Since $\psi: \bigwedge^{d} V_1 \xrightarrow{\sim} \bigwedge^{d} V_2$ is an isomorphism of $W$-modules, we have
  \[\psi \Bigl(\bigcap_{1 \le j \le d} V_{1,d,i_j}^-\Bigr) = \bigcap_{1 \le j \le d} V_{2,d,i_j}^-\]
  for any set of indices $1 \le i_1 < \dots < i_d \le n$, and hence
  \[\bigwedge^d V_2 = \bigoplus_{1 \le i_1 < \dots < i_d \le n} \Bigl( \bigcap_{1 \le j \le d} V_{2,d,i_j}^- \Bigr)\]
  as claimed.
\end{proof}

Recall that the number $d$ satisfies $d+1 \le n$.
We have the following lemma which is a  ``weak version'' of Proposition \ref{prop-tran-lin-ind}.

\begin{lemma} \label{lem-corr-basis}
  Suppose $1 \le j_1, \dots, j_{d+1} \le k$.
  If $\alpha_{j_1}, \dots, \alpha_{j_{d+1}}$ are linearly independent, then so are $\beta_{j_1}, \dots, \beta_{j_{d+1}}$.
\end{lemma}

\begin{proof}
  Suppose otherwise that $\beta_{j_1}, \dots, \beta_{j_{d+1}}$ are linearly dependent and the subset $\{\beta_{j_1}, \dots, \beta_{j_{h}}\}$ ($h \le d$) is a maximal linearly independent set.
  Then there exists a nonzero vector $v$ in $\bigwedge^d V_2$ of the form $v = \beta_{j_1} \wedge \dots \wedge \beta_{j_{h}} \wedge v_{h+1} \wedge \dots \wedge v_d$.
  By Lemma \ref{lem-intersect-eigenspace}, we have
  \[v \in \bigcap_{1 \le i \le h} V_{2,d,j_i}^-.\]
  Note that for any index $i$ such that $h+1 \le i \le d+1$, $\beta_{j_i}$ is a linear combination of $\beta_{j_1}, \dots, \beta_{j_{h}}$.
  Then by Lemma \ref{lem-extra-eigen}, we have
  \[v \in \bigcap_{1 \le i \le h} V_{2,d,j_i}^- = \bigcap_{1 \le i \le d+1} V_{2,d,j_i}^-.\]
  As in the proof of Lemma \ref{lem-com-eigen-decomp}, since $\psi: \bigwedge^{d} V_1 \xrightarrow{\sim} \bigwedge^{d} V_2$ is an isomorphism of $W$-modules, we have
  \[\psi^{-1} (v) \in \bigcap_{1 \le i \le d+1} V_{1,d,j_i}^-.\]
  Note that $\psi^{-1} (v)$ is a nonzero vector.
  However, the intersection $\bigcap_{1 \le i \le d+1} V_{1,d,j_i}^-$ is zero by Lemma \ref{lem-intersect-eigenspace}.
  This is absurd.
\end{proof}

\begin{proposition} \label{prop-corr-basis}
  Suppose $1 \le j_1, \dots, j_n \le k$.
  If $\{\alpha_{j_1}, \dots, \alpha_{j_n}\}$  is a basis of $V_1$, then $\{\beta_{j_1}, \dots, \beta_{j_n}\}$ is a basis of $V_2$, and vice versa.
\end{proposition}

\begin{proof}
  Without loss of generality, we may assume $j_1 = 1$, $j_2 = 2$, $\dots$, $j_n = n$.
  Then $\{\alpha_1, \dots, \alpha_n\}$ is a basis of $V_1$.
  Suppose the space $U := \mathbb{F} \langle \beta_1, \dots, \beta_n \rangle$ spanned by $\beta_1, \dots, \beta_n$ is a proper subspace of $V_2$, and $m := \dim U < n$.
  We may assume further that $\{\beta_1, \dots, \beta_m\}$ is a basis of $U$.

  For any indices $1 \le i_1 < \dots < i_d \le n$, there exists an index $i_{d+1} \in [n] \setminus \{i_1, \dots, i_d\}$ since $d \le n-1$.
  Note that the vectors $\alpha_{i_1}, \dots, \alpha_{i_d}, \alpha_{i_{d+1}}$ are linearly independent.
  By Lemma \ref{lem-corr-basis}, $\beta_{i_1}, \dots, \beta_{i_d}, \beta_{i_{d+1}} \in U$ are linearly independent as well.
  In particular, $\beta_{i_1} \wedge \dots \wedge \beta_{i_d} \ne 0$ and we have by Lemma \ref{lem-intersect-eigenspace} that
  \[\bigcap_{1 \le j \le d} V_{2, d, i_j}^- = \mathbb{F} \langle \beta_{i_1} \wedge \dots \wedge \beta_{i_d} \rangle \subseteq \bigwedge^d U. \]
  But then by Lemma \ref{lem-com-eigen-decomp} we have
  \[\bigwedge^d V_2 = \bigoplus_{1 \le i_1 < \dots < i_d \le n} \Bigl( \bigcap_{1 \le j \le d} V_{2,d,i_j}^- \Bigr) \subseteq \bigwedge^d U \subsetneq \bigwedge^d V_2\]
  which is a contradiction.
  Thus, we must have $U = V_2$, $m = n$, and $\{\beta_1, \dots, \beta_n\}$ is a basis of $V_2$.
\end{proof}

As a corollary, we have the following proposition.

\begin{proposition} \label{prop-tran-lin-ind}
  Suppose $h \le n$ and  $1 \le i_1, \dots, i_h \le k$.
  If $\alpha_{i_1}, \dots, \alpha_{i_h}$ are linearly independent, then so are $\beta_{i_1}, \dots, \beta_{i_h}$.
  In particular, if $\alpha_i$ and $\alpha_j$ are not proportional for some $1 \le i \ne j \le k$, then so are $\beta_i$, $\beta_j$.
\end{proposition}

\begin{proof}
  Recall Corollary \ref{cor-span-refl} that $V_1$ is spanned by $\alpha_1, \dots, \alpha_k$.
  Thus there exist reflection vectors $\alpha_{i_{h+1}}, \alpha_{i_{h+2}}, \dots, \alpha_{i_{n}}$ such that $\alpha_{i_1}, \dots, \alpha_{i_h}, \alpha_{i_{h+1}}, \dots, \alpha_{i_{n}}$ form a basis of $V_1$.
  Then use Proposition \ref{prop-corr-basis}.
\end{proof}

\subsection{Coincidence of the associated digraphs} \label{subsec-coinc-digraphs}

Recall in Definition \ref{def-assoc-digr} that a digraph is associated with any subset $I \subseteq [k]$ and any reflection representation.
For $\iota = 1, 2$, we denote temporarily by $G_\iota$  the associated graph to the full set $[k]$ and the representation $(V_\iota,\rho_\iota)$.
In this subsection we will prove that $G_1 = G_2$.

For two distinct indices $i, j \in [k]$, we set
\begin{alignat*}{2}
  s_i \cdot \alpha_j & = \alpha_j + x_{ji} \alpha_i, \quad & x_{ji} & \in \mathbb{F},  \\
  s_i \cdot \beta_j & = \beta_j + y_{ji} \beta_i, & y_{ji} & \in \mathbb{F}.
\end{alignat*}
Then $j \to i$ is an arrow in $G_1$, $G_2$ if and only if $x_{ji}$, $y_{ji} \ne 0$, respectively.

For distinct indices $1 \le i_1, \dots, i_d \le k$, if $\alpha_{i_1} \wedge \dots \wedge \alpha_{i_d} \ne 0$, that is, if $\alpha_{i_1}, \dots, \alpha_{i_d}$ are linearly independent, then by Proposition \ref{prop-tran-lin-ind}, the vectors $\beta_{i_1}, \dots, \beta_{i_d}$ are linearly independent as well.
Moreover, we have by Lemma \ref{lem-intersect-eigenspace}
\[\bigcap_{1 \le j \le d} V_{1,d,i_j}^- = \mathbb{F} \langle \alpha_{i_1} \wedge \dots \wedge \alpha_{i_d} \rangle, \quad \text{and} \quad \bigcap_{1 \le j \le d} V_{2,d,i_j}^- = \mathbb{F} \langle \beta_{i_1} \wedge \dots \wedge \beta_{i_d} \rangle.\]
Therefore, since $\psi: \bigwedge^d V_1 \xrightarrow{\sim} \bigwedge^d V_2$ is an isomorphism of $W$-modules, it holds
\[\psi(\alpha_{i_1} \wedge \dots \wedge \alpha_{i_d}) = \zeta_{i_1, \dots, i_d} \beta_{i_1} \wedge \dots \wedge \beta_{i_d}, \quad \text{for some } \zeta_{i_1, \dots, i_d} \in \mathbb{F}^\times.\]
By convention, we define $\zeta_{i_1, \dots, i_d} := 0$ if $\alpha_{i_1} \wedge \dots \wedge \alpha_{i_d} = 0$.

\begin{remark} \label{rmk-zeta-inde}
  Note that the coefficients $\zeta_{i_1, \dots, i_d}$ are independent of the order of $i_1, \dots, i_d$, that is, $\zeta_{i_1, \dots, i_d} = \zeta_{i_{\sigma(1)}, \dots, i_{\sigma(d)}}$ for any permutation $\sigma \in \mathfrak{S}_d$ (as $\gamma_{i_1, \dots, i_d}$ in Section \ref{sec-main1}).
\end{remark}

\begin{lemma} \label{lem-same-digraph}
  Suppose $i, j \in [k]$ and $i \ne j$.
  There exist distinct indices
  $i_2, \dots, i_{d} \in [k]$
  such that
  \begin{enumerate}
    \item \label{lem-same-digraph-con1} $\alpha_{i}, \alpha_j, \alpha_{i_2}, \dots, \alpha_{i_{d}}$ are linearly independent if $\alpha_i, \alpha_j$ are not proportional;
    \item \label{lem-same-digraph-con2} $\alpha_i, \alpha_{i_2}, \dots, \alpha_{i_{d}}$ are linearly independent if $\alpha_i, \alpha_j$ are proportional (thus in this case the vectors $\alpha_j, \alpha_{i_2}, \dots, \alpha_{i_{d}}$ are linearly independent as well).
  \end{enumerate}
\end{lemma}

\begin{proof}
  The existence of the required indices is ensured by the facts that $d+1 \le n$ and that $V_1$ is spanned by all the reflection vectors (Corollary \ref{cor-span-refl}).
\end{proof}

The following lemma is the key in this subsection.

\begin{lemma} \label{lem-same-digraph-2}
  Suppose $i, j \in [k]$ and $i \ne j$.
  Let  $i_2,  \dots, i_{d} \in [k]$ be any indices satisfying the conditions \eqref{lem-same-digraph-con1}\eqref{lem-same-digraph-con2} in Lemma \ref{lem-same-digraph}.
  Then we have
  \begin{equation}\label{eq-lem-same-digraph-2}
    \zeta_{i, i_2, \dots, i_{d}} y_{i j} = \zeta_{j, i_2, \dots, i_{d}} x_{i j}.
  \end{equation}
\end{lemma}

\begin{proof}
  We consider
  \begin{align}
    & \mathrel{\phantom{=}} s_j \cdot (\psi (\alpha_i \wedge \alpha_{i_2} \wedge \dots \wedge \alpha_{i_{d}})) \notag \\
    & = \zeta_{i, i_2, \dots, i_{d}} s_j \cdot (\beta_i \wedge \beta_{i_2} \wedge \dots \wedge \beta_{i_{d}}) \notag \\
    & = \zeta_{i, i_2, \dots, i_{d}} (\beta_i + y_{ij} \beta_j) \wedge (\beta_{i_2} + y_{i_2 j} \beta_j) \wedge \dots \wedge (\beta_{i_d} + y_{i_d j} \beta_j) \notag \\
    & = \zeta_{i, i_2, \dots, i_{d}} (\beta_i \wedge \beta_{i_2} \wedge \dots \wedge \beta_{i_{d}} + y_{ij} \beta_j \wedge \beta_{i_2} \wedge \dots \wedge \beta_{i_{d}}) \label{eq-lem-same-digraph-2-1} \\
    & \mathrel{\phantom{=}} + \sum_{2 \le l \le d} (-1)^{d-l} \zeta_{i, i_2, \dots, i_d} y_{i_l j} \beta_i \wedge \beta_{i_2} \wedge \dots \wedge \widehat{\beta}_{i_l} \wedge \dots \wedge \beta_{i_d} \wedge \beta_{j}\notag
  \end{align}
  which also equals
  \begin{align}
    & \mathrel{\phantom{=}} \psi (s_j \cdot (\alpha_i \wedge \alpha_{i_2} \wedge \dots \wedge \alpha_{i_d})) \notag \\
    & = \psi ((\alpha_i + x_{ij} \alpha_j) \wedge (\alpha_{i_2} + x_{i_2 j} \alpha_j) \wedge \dots \wedge (\alpha_{i_d} + x_{i_d j} \alpha_j)) \notag \\
    & = \psi \Bigl(\alpha_i \wedge \alpha_{i_2} \wedge \dots \wedge \alpha_{i_d} + x_{ij} \alpha_j \wedge \alpha_{i_2} \wedge \dots \wedge \alpha_{i_d} \notag \\
    & \mathrel{\phantom{=}} \phantom{\psi \Bigl(} + \sum_{2 \le l \le d} (-1)^{d-l} x_{i_l j} \alpha_i \wedge \alpha_{i_2} \wedge \dots \wedge \widehat{\alpha}_{i_l} \wedge \dots \wedge \alpha_{i_d} \wedge \alpha_{j} \Bigr) \notag \\
    & = \zeta_{i, i_2, \dots, i_d} \beta_i \wedge \beta_{i_2} \wedge \dots \wedge \beta_{i_d} + \zeta_{j, i_2, \dots, i_d} x_{ij} \beta_j \wedge \beta_{i_2} \wedge \dots \wedge \beta_{i_d} \label{eq-lem-same-digraph-2-2} \\
    & \mathrel{\phantom{=}} + \sum_{2 \le l \le d} (-1)^{d-l} \zeta_{i, i_2, \dots, \widehat{i}_l, \dots, i_d, j} x_{i_l j} \beta_i \wedge \beta_{i_2} \wedge \dots \wedge \widehat{\beta}_{i_l} \wedge \dots \wedge \beta_{i_d} \wedge \beta_j.  \notag
  \end{align}

  If $\alpha_i, \alpha_j$ are not proportional, then $\alpha_i, \alpha_j, \alpha_{i_2}, \dots, \alpha_{i_d}$ are linearly independent by our assumption, and then so are $\beta_i, \beta_j, \beta_{i_2}, \dots, \beta_{i_d}$ by Proposition \ref{prop-tran-lin-ind} (or Lemma \ref{lem-corr-basis}).
  Therefore, the vectors occurring in \eqref{eq-lem-same-digraph-2-1} and \eqref{eq-lem-same-digraph-2-2} are nonzero and linearly independent.
  By comparing the coefficients of $\beta_j \wedge \beta_{i_2} \wedge \dots \wedge \beta_{i_d}$ in \eqref{eq-lem-same-digraph-2-1} and \eqref{eq-lem-same-digraph-2-2} we see that the desired Equation \eqref{eq-lem-same-digraph-2} holds.

  If $\alpha_i, \alpha_j$ are proportional, then so are $\beta_i, \beta_j$ by Proposition \ref{prop-tran-lin-ind}, and the summations in \eqref{eq-lem-same-digraph-2-1} and \eqref{eq-lem-same-digraph-2-2} vanish, that is,
  \begin{align*}
    \eqref{eq-lem-same-digraph-2-1} & = \zeta_{i, i_2, \dots, i_d} \beta_i \wedge \beta_{i_2} \wedge \dots \wedge \beta_{i_d} + \zeta_{i, i_2, \dots, i_d} y_{ij} \beta_j \wedge \beta_{i_2} \wedge \dots \wedge \beta_{i_d}, \\
    \eqref{eq-lem-same-digraph-2-2} & = \zeta_{i, i_2, \dots, i_d} \beta_i \wedge \beta_{i_2} \wedge \dots \wedge \beta_{i_d} + \zeta_{j, i_2, \dots, i_d} x_{ij} \beta_j \wedge \beta_{i_2} \wedge \dots \wedge \beta_{i_d}.
  \end{align*}
  As pointed out in Lemma \ref{lem-same-digraph}, the vectors $\alpha_j, \alpha_{i_2}, \dots, \alpha_{i_d}$ are linearly independent, and so are $\beta_j, \beta_{i_2}, \dots, \beta_{i_d}$ by Proposition \ref{prop-tran-lin-ind}.
  Therefore, we have Equation \eqref{eq-lem-same-digraph-2} again by comparing the two equations above.
\end{proof}

Note that in Equation \eqref{eq-lem-same-digraph-2} the coefficients $\zeta_{i, i_2, \dots, i_d}$ and $\zeta_{j, i_2, \dots, i_d}$ are nonzero.
Therefore we have the following corollary.

\begin{corollary} \label{cor-x-y}
  Suppose $i, j \in [k]$ and $i \ne j$.
  Then $x_{ij}$ and $y_{ij}$ are equal or not equal to zero simultaneously, that is, either $x_{ij} = y_{ij} = 0$ or $x_{ij} y_{ij} \ne 0$.
\end{corollary}

By the definition of the associated digraphs $G_1$ and $G_2$, Corollary \ref{cor-x-y} implies

\begin{corollary}
  $G_1 = G_2$.
\end{corollary}

From now on, we recover the notation $G = G_{[k]}$ to indicate uniformly the digraphs $G_1$ and $G_2$, and $G_I$ to be the sub-digraph spanned by a subset $I \subseteq [k]$ (see Definition \ref{def-assoc-digr}).

We will also need the following corollary of Lemma \ref{lem-same-digraph-2}.

\begin{corollary} \label{cor-same-digraph-2}
  Let $i, j, i_2, \dots, i_d \in [k]$ be as in Lemma \ref{lem-same-digraph-2}.
  Suppose $x_{ij} \ne 0$ and $y_{ji} \ne 0$ (equivalently, $y_{ij} \ne 0$ and $x_{ji} \ne 0$).
  Then we have
  \[\frac{y_{ij}}{x_{ij}}  = \frac{\zeta_{j, i_2, \dots, i_d}}{\zeta_{i, i_2, \dots, i_d}} = \frac{x_{ji}}{y_{ji}}.\]
\end{corollary}

\begin{proof}
  The first desired equality is nothing but Equation \eqref{eq-lem-same-digraph-2}.
  By swapping the indices $i$ and $j$ in Equation \eqref{eq-lem-same-digraph-2}, we obtain the second equality.
\end{proof}

\subsection{The linear isomorphism \texorpdfstring{$f$}{f} from \texorpdfstring{$V_1$}{V1} to \texorpdfstring{$V_2$}{V2}} \label{subsec-f}

Remember that our final goal is to find an isomorphism $f : V_1 \xrightarrow{\sim} V_2$ of $W$-modules.
For this, let us introduce some notations.

By applying Lemma \ref{lem-GI} to the reflection representation $(V_1, \rho_1)$, we choose and fix a subset $I \subseteq [k]$ such that
\begin{enumerate}
  \item $G_I$ is weakly connected, and
  \item $\{\alpha_i \mid i \in I\}$ is a basis of $V_1$.
\end{enumerate}
Then by Proposition \ref{prop-corr-basis}, $\{\beta_i \mid i \in I\}$ is a basis of $V_2$.

For two indices $i, j \in I$ such that  $i \ne j$ and either $i \to j$ or $j \to i$ is an arrow in $G_I$, we define
\[z_{ij} :=
    \begin{cases}
      \frac{y_{ij}}{x_{ij}}, & \text{if $i \to j$ is an arrow} \\
      \frac{x_{ji}}{y_{ji}}, & \text{if $j \to i$ is an arrow}.
    \end{cases}\]
By Corollary \ref{cor-same-digraph-2}, we have $\frac{y_{ij}}{x_{ij}} = \frac{x_{ji}}{y_{ji}}$ if both $i \to j$ and $j \to i$ are arrows.
So the element $z_{ij} \in \mathbb{F}^\times$ is well defined.
We have the following lemma.

\begin{lemma} \label{lem-prod-zij}
  Let $h \ge 1$ be an integer and
  \[i_0, \quad a_1, \quad i_1, \quad a_2, \quad i_2, \quad \dots \quad i_{h-1}, \quad a_h, \quad i_h\]
  be an undirected walk in $G_I$.
  For any $p \in [h]$ and any distinct indices $j_2, \dots, j_{d} \in I \setminus \{i_0, i_p\}$ (if $i_0 = i_p$ then we regard $\{i_0, i_p\} = \{i_0\}$) we have
    \begin{equation}\label{eq-claim}
      z_{i_0 i_1} \cdots z_{i_{p-1} i_p} = \frac{\zeta_{i_p, j_2, \dots, j_d}}{\zeta_{i_0, j_2, \dots, j_d}}.
    \end{equation}
\end{lemma}

Note that  $\lvert I \setminus \{i_0, i_p\} \rvert \ge n - 2 \ge d - 1$.
Therefore such indices $j_2, \dots, j_{d}$ exist.
Note also that $\{\alpha_j \mid j \in I\}$ is a basis for $V_1$.
Thus $\zeta_{i_0, j_2, \dots, j_d}$ and $\zeta_{i_p, j_2, \dots, j_d}$ are nonzero.

\begin{proof}
  We prove by induction on $p$.
  Suppose first that $p = 1$.
  Then the desired equality $z_{i_0 i_1} = \zeta_{i_1, j_2, \dots, j_d} / \zeta_{i_0, j_2, \dots, j_d}$ follows from Corollary \ref{cor-same-digraph-2}.

  Suppose now $p \ge 2$.
  The induction hypothesis reads
  \[z_{i_0 i_1} \cdots z_{i_{p-2} i_{p-1}} = \frac{\zeta_{i_{p-1}, j_2, \dots, j_d}}{\zeta_{i_0, j_2, \dots, j_d}} \text{ for any distinct $j_2, \dots, j_d \in I \setminus \{i_0, i_{p-1}\}$.}\]
  For distinct indices $j_2, \dots, j_d \in I \setminus \{i_0, i_p\}$ given arbitrarily, we have three cases.

  \emph{Case one}: $i_{p-1} \notin \{j_2, \dots, j_d\}$.
  By the same arguments as in the beginning case ``$p = 1$'', we have
  \[z_{i_{p-1} i_p} = \frac{\zeta_{i_p, j_2, \dots, j_d}}{\zeta_{i_{p-1}, j_2, \dots, j_d}}\]
  whenever $a_p = i_{p-1} \to i_p$ or $a_p = i_p \to i_{p-1}$.
  Therefore by induction hypothesis we have
  \[z_{i_0 i_1} \cdots z_{i_{p-2} i_{p-1}} z_{i_{p-1} i_p} = \frac{\zeta_{i_{p-1}, j_2, \dots, j_d}}{\zeta_{i_0, j_2, \dots, j_d}} \cdot \frac{\zeta_{i_p, j_2, \dots, j_d}}{\zeta_{i_{p-1}, j_2, \dots, j_d}} = \frac{\zeta_{i_p, j_2, \dots, j_d}}{\zeta_{i_0, j_2, \dots, j_d}} \]
  as claimed in Equation \eqref{eq-claim}.

  \emph{Case two}: $i_{p-1} \in \{j_2, \dots, j_d\}$ and $i_0 \ne i_p$.
  We may assume $i_{p-1} = j_2$.
  Then $i_0$, $i_p$, $i_{p-1}$ ($= j_2$), $j_3$, $\dots$, $j_d$ are distinct.
  We have
  \[z_{i_0 i_1} \cdots z_{i_{p-2} i_{p-1}} = \frac{\zeta_{i_{p-1}, i_p, j_3, \dots, j_d}}{\zeta_{i_0, i_p, j_3, \dots, j_d}}\]
  by applying induction hypothesis to the indices $i_p, j_3, \dots, j_d$, and
  \[z_{i_{p-1} i_p} = \frac{\zeta_{i_p, i_0, j_3, \dots, j_d}}{\zeta_{i_{p-1}, i_0, j_3, \dots, j_d}}\]
  by the same  arguments as in the beginning case ``$p=1$''.
  Therefore,
  \begin{equation} \label{eq-claim-1}
    z_{i_0 i_1} \cdots z_{i_{p-2} i_{p-1}} z_{i_{p-1} i_p} = \frac{\zeta_{i_{p-1}, i_p, j_3, \dots, j_d}}{\zeta_{i_0, i_p, j_3, \dots, j_d}} \cdot \frac{\zeta_{i_p, i_0, j_3, \dots, j_d}}{\zeta_{i_{p-1}, i_0, j_3, \dots, j_d}}.
  \end{equation}
  Note that $\zeta_{i_0, i_p, j_3, \dots, j_d} = \zeta_{i_p, i_0, j_3, \dots, j_d}$ (see Remark \ref{rmk-zeta-inde}).
  Therefore Equation \eqref{eq-claim-1} reduces to
  \[z_{i_0 i_1} \cdots z_{i_{p-2} i_{p-1}} z_{i_{p-1} i_p} = \frac{\zeta_{i_{p-1}, i_p, j_3, \dots, j_d}}{\zeta_{i_{p-1}, i_0, j_3, \dots, j_d}} = \frac{\zeta_{i_p, j_2, j_3, \dots, j_d}}{\zeta_{i_0, j_2, j_3, \dots, j_d}}\]
  which is what we want.

  \emph{Case three}: $i_{p-1} \in \{j_2, \dots, j_d\}$ and $i_0 = i_p$.
  Then our goal Equation \eqref{eq-claim} becomes
  \begin{equation}\label{eq-claim-2}
    z_{i_0 i_1} \cdots z_{i_{p-1} i_p} = 1.
  \end{equation}
  We can still assume $i_{p-1} = j_2$.
  Note that in this case $i_0$ ($= i_p$), $i_{p-1}$ ($= j_2$), $j_3, \dots, j_d$ are distinct $d$ indices.
  But $d \le n-1$, so there exists an extra index $j \in I \setminus \{i_0, j_2, j_3, \dots, j_d\}$.
  Similar to the former cases, we have
  \[z_{i_0 i_1} \cdots z_{i_{p-2} i_{p-1}} = \frac{\zeta_{i_{p-1}, j, j_3, \dots, j_d}}{\zeta_{i_0, j, j_3, \dots, j_d}}\]
  and
  \[z_{i_{p-1} i_p} = \frac{\zeta_{i_p, j, j_3, \dots, j_d}}{\zeta_{i_{p-1}, j, j_3, \dots, j_d}}.\]
  Therefore, we have
  \[z_{i_0 i_1} \cdots z_{i_{p-2} i_{p-1}} z_{i_{p-1} i_p} = \frac{\zeta_{i_{p-1}, j, j_3, \dots, j_d}}{\zeta_{i_0, j, j_3, \dots, j_d}} \cdot \frac{\zeta_{i_p, j, j_3, \dots, j_d}}{\zeta_{i_{p-1}, j, j_3, \dots, j_d}} = \frac{\zeta_{i_p, j, j_3, \dots, j_d}}{\zeta_{i_0, j, j_3, \dots, j_d}} = 1\]
  which is exactly Equation \eqref{eq-claim-2}.
\end{proof}

Now we are ready to construct a linear map $f$ from $V_1$ to $V_2$.
Such a map is determined by vectors $\{f(\alpha_i) \in V_2 \mid i \in I\}$ since $\{\alpha_i \in V_1 \mid i \in I\}$ is a basis of $V_1$.
Because $\alpha_i$ and $\beta_i$ are the reflection vectors of $s_i$ in $V_1$ and $V_2$ respectively, we are excepted to have
\[f(\alpha_i) = z_i \beta_i, \text{ for some } z_i \in \mathbb{F}^\times \text{ and each } i \in I.\]
Below we propose a choice of the coefficients $z_i$.

From now on we fix an index $i_0 \in I$ and set $z_{i_0} := 1$.
For any other index $i \in I$, we choose an undirected walk in $G_I$ from $i_0$ to $i$, say,
\begin{equation*}
  i_0, \quad  a_1, \quad  i_1, \quad  a_2, \quad  i_2, \quad \dots, \quad  i_{l-1}, \quad  a_l, \quad   i_l  = i,
\end{equation*}
where $i_0, i_1, \dots, i_l \in I$.
Then $z_i$ is defined to be
\[z_i := z_{i_0 i_1} z_{i_1 i_2} \cdots z_{i_{l-1} i_l} \in \mathbb{F}^\times.\]
We need to show that $z_i$ such defined is independent of the choice of the undirected walk (but it does depend on the choice of the beginning vertex $i_0$).

\begin{proposition} \label{prop-zi}
  Fix $i_0 \in I$ as above.
  For each $i \in I$, the value of $z_i$ only depends on $i$, not on the choice of the undirected walk from $i_0$ to $i$.
\end{proposition}

\begin{proof}
  Suppose there exist two undirected walks in $G_I$ from $1$ to $i$, say ($h > l$),
  \begin{alignat}{9}
    &  i_0, \quad & & a_1, \quad & & i_1, \quad & & a_2, \quad & & i_2, \quad & & \dots, \quad & & i_{l-1}, \quad & & a_l, \quad &  i_l & = i, \label{eq-walk-1} \\
    &  i_0, \quad & & a_h, \quad & & i_{h-1}, \quad & & a_{h-1}, \quad & & i_{h-2}, \quad & & \dots, \quad & & i_{l+1}, \quad & & a_{l+1}, \quad & i_l & = i. \label{eq-walk-2}
  \end{alignat}
  By convention, we also denote $i_{h} := i_0$.
  We need to show
  \[z_{i_0 i_1} z_{i_1 i_2} \cdots z_{i_{l-1} i_l} = z_{i_h i_{h-1}} z_{i_{h-1} i_{h-2}} \cdots z_{i_{l+1} i_l}.\]
  By definition we have $z_{i_{j-1} i_j} = z_{i_j i_{j-1}}^{-1}$.
  Thus our goal becomes
  \[z_{i_0 i_1} z_{i_1 i_2} \cdots z_{i_{l-1} i_l} z_{i_l i_{l+1}} \cdots z_{i_{h-2} i_{h-1}} z_{i_{h-1} i_h} = 1.\]
  By extending the first undirected walk \eqref{eq-walk-1} by the reverse of \eqref{eq-walk-2}, we have an undirected walk in $G_I$ from $i_0$ to $i_0$,
  \[i_0, \quad a_1, \quad i_1, \quad \dots, \quad a_{l-1}, \quad i_l, \quad a_{l+1}, \quad \dots, \quad i_{h-1}, \quad a_h, \quad i_h = i_0.\]
  By taking $p = h$ in Lemma \ref{lem-prod-zij}, we have that
  \[z_{i_0 i_1} \cdots z_{i_{h-1} i_h} = \frac{\zeta_{i_h, j_2, \dots, j_d}}{\zeta_{i_0, j_2, \dots, j_d}} = 1\]
  for any distinct indices $j_1, \dots, j_d \in I \setminus \{i_0, i_h\}$.
  This equality is what we want.
\end{proof}

By Proposition \ref{prop-zi}, we have a well defined linear map $f: V_1 \to V_2$ by setting
\[f(\alpha_i) := z_i \beta_i, \quad \text{for each } i \in I.\]
Since $z_i \in \mathbb{F}^\times$ and $\{\alpha_i \mid i \in I\}$, $\{\beta_i \mid i \in I\}$ are bases of $V_1$, $V_2$ respectively, the map $f$ is clearly an  isomorphism of vector spaces.
It remains to show that $f$ is a homomorphism of $W$-modules.

\subsection{The map \texorpdfstring{$f$}{f} is a \texorpdfstring{$W$}{W}-isomorphism} \label{subsec-isom}

To show that $f$ is a homomorphism of $W$-modules, it suffices to show
\[f(s_h \cdot \alpha_i) = s_h \cdot f(\alpha_i), \quad \text{for any } h \in [k] \text{ and } i \in I.\]
We split the proof into two parts (Propositions \ref{prop-W-map-1} and \ref{prop-W-map-2}), depending on whether $h \in I$ or $h \notin I$.

\begin{proposition} \label{prop-W-map-1}
  For any $h, i \in I$, we have $f(s_h \cdot \alpha_i) = s_h \cdot f(\alpha_i)$.
\end{proposition}

\begin{proof}
  In this case we have
  \[f(s_h \cdot \alpha_i) = f (\alpha_i + x_{ih} \alpha_h) = z_i \beta_i + x_{ih} z_h \beta_h,\]
  and
  \[s_h \cdot f(\alpha_i) = s_h \cdot (z_i \beta_i) = z_i \beta_i + z_i y_{ih} \beta_h.\]
  Therefore, we need to show
  \begin{equation}\label{eq-W-map-1}
    x_{ih} z_h = z_i y_{ih}.
  \end{equation}

  If $x_{ih} = 0$, then $y_{ih} = 0$ by Corollary \ref{cor-x-y}, and thus Equation \eqref{eq-W-map-1} holds trivially.
  Suppose otherwise that $x_{ih} \ne 0$.
  Then $i \to h$ is an arrow.
  Recall that the coefficient $z_i$ is computed by taking arbitrarily an undirected walk in $G_I$ from the fixed $i_0$ ($\in I$) to $i$, say,
  \begin{equation*}
    i_0, \quad  a_1, \quad i_1, \quad  \dots, \quad  a_l, \quad   i_l  = i.
  \end{equation*}
  If we set $i_{l+1} = h$ and $a_{l+1} = i \to h$, then the extended undirected walk
  \begin{equation*}
    i_0, \quad  a_1, \quad  i_1, \quad  \dots, \quad  a_l, \quad i_l = i, \quad a_{l+1}, \quad i_{l+1} = h
  \end{equation*}
  goes from $i_0$ to $h$.
  Therefore, by the definitions of $z_i$ and $z_{ih}$ in Subsection \ref{subsec-f}, we have $z_h = z_i z_{ih} = z_i y_{ih} / x_{ih}$  which is exactly Equation \eqref{eq-W-map-1}.
\end{proof}

\begin{proposition} \label{prop-W-map-2}
  For any $i \in I$ and $h \in [k] \setminus I$, we have $f(s_h \cdot \alpha_i) = s_h \cdot f(\alpha_i)$.
\end{proposition}

The rest of this subsection is devoted to proving Proposition \ref{prop-W-map-2}.
Since $\{\alpha_l \mid l \in I\}$ is a basis for $V_1$, and so is $\{\beta_l \mid l \in I\}$ for $V_2$ by Proposition \ref{prop-corr-basis}, we write
\[\alpha_h = \sum_{l \in I} a_l \alpha_l, \quad \beta_h = \sum_{l \in I} b_l \beta_l, \quad \text{where } a_l, b_l \in \mathbb{F}.\]
Then we have
\[f(s_h \cdot \alpha_i) = f(\alpha_i + x_{ih} \alpha_h) = f\Bigl( \alpha_i + x_{ih} \sum_{l \in I} a_l \alpha_l \Bigr) = z_i \beta_i + \sum_{l \in I } x_{ih} a_l z_l \beta_l \]
and
\[s_h \cdot f(\alpha_i) = s_h \cdot (z_i \beta_i) = z_i \beta_i + z_i y_{ih} \beta_h = z_i \beta_i + \sum_{l \in I} z_i y_{ih} b_l \beta_l.\]
Therefore, to prove Proposition \ref{prop-W-map-2} it suffices to show for any $j \in I$ that
\begin{equation}\label{eq-W-map-2}
  x_{ih} a_j z_j = y_{ih} b_j z_i.
\end{equation}
For this, we need the following lemma.

\begin{lemma} \label{lem-W-map}
  Suppose $h \in [k] \setminus I$, $j \in I$.
  We write $\alpha_h = \sum_{l \in I} a_l \alpha_l$, $\beta_h = \sum_{l \in I} b_l \beta_l$ as above.
  If $a_j \ne 0$, then for any distinct $i_2, \dots, i_d \in I \setminus \{j\}$,
  we have
  \[\zeta_{j, i_2, \dots, i_d} a_j = \zeta_{h, i_2, \dots, i_d} b_j.\]
\end{lemma}

Note that both $\zeta_{j, i_2, \dots, i_d}$ and $\zeta_{h, i_2, \dots, i_d}$ are nonzero, because both of the two sets $\{\alpha_j, \alpha_{i_2}, \dots, \alpha_{i_d}\}$ and $\{\alpha_h, \alpha_{i_2}, \dots, \alpha_{i_d}\}$ are linearly independent sets.

\begin{proof}
  Since $a_j \ne 0$, the vectors $\alpha_h, \alpha_{i_2}, \dots, \alpha_{i_d}$ are linearly independent.
  We write $I = \{j, i_2, \dots, i_d, i_{d+1}, \dots, i_n\}$.
  Then
  \begin{equation}\label{eq-lem-W-map}
    \alpha_h \wedge \alpha_{i_2} \wedge \dots \wedge \alpha_{i_d} = \Bigl( a_j \alpha_j + \sum_{d+1 \le l \le n} a_{i_l} \alpha_{i_l} \Bigr) \wedge \alpha_{i_2} \wedge \dots \wedge \alpha_{i_d}.
  \end{equation}
  The image under $\psi$ of the left hand side of Equation \eqref{eq-lem-W-map} is
  \begin{align}
    \psi(\alpha_h \wedge \alpha_{i_2} \wedge \dots \wedge \alpha_{i_d}) & = \zeta_{h, i_2, \dots, i_d} \beta_h \wedge \beta_{i_2} \wedge \dots \wedge \beta_{i_d} \notag \\
     & = \zeta_{h, i_2, \dots, i_d} \Bigl( b_j \beta_j + \sum_{d+1 \le l \le n} b_{i_l} \beta_{i_l} \Bigr) \wedge \beta_{i_2} \wedge \dots \wedge \beta_{i_d}. \label{eq-lem-W-map-2}
  \end{align}
  Moreover, the image under $\psi$ of the right hand side of Equation \eqref{eq-lem-W-map} equals
  \begin{equation}\label{eq-lem-W-map-3}
    \zeta_{j, i_2, \dots, i_d} a_j \beta_j \wedge \beta_{i_2} \wedge \dots \wedge \beta_{i_d} + \sum_{d+1 \le l \le n} \zeta_{i_l, i_2, \dots, i_d} a_{i_l} \beta_{i_l} \wedge \beta_{i_2} \wedge \dots \wedge \beta_{i_d}.
  \end{equation}
  The equality of \eqref{eq-lem-W-map-2} and \eqref{eq-lem-W-map-3} gives $\zeta_{j, i_2, \dots, i_d} a_j = \zeta_{h, i_2, \dots, i_d} b_j$.
\end{proof}

Now we are ready to complete the proof of Proposition \ref{prop-W-map-2}.

\begin{proof}[Proof of Proposition \ref{prop-W-map-2}]
  As we mentioned, it suffices to prove Equation \eqref{eq-W-map-2} for any $j \in I$.
  We have three cases.

  \emph{Case one}: $a_j = 0$.
  Then $\{\alpha_h\} \cup \{\alpha_l \mid l \in I \setminus \{j\}\}$ is a linearly dependent set.
  Then $\{\beta_h\} \cup \{\beta_l \mid l \in I \setminus \{j\}\}$ is also linearly dependent.
  Otherwise, it would be a basis for $V_2$, contradicting Proposition \ref{prop-corr-basis}.
  Therefore, we have $b_j = 0$, and hence Equation \eqref{eq-W-map-2} holds trivially.

  \emph{Case two}: $a_j \ne 0$ and $j = i$.
  In this case Equation \eqref{eq-W-map-2} reduces to
  \[x_{jh} a_j = y_{jh} b_j.\]
  If $\alpha_h$ and $\alpha_j$ are proportional, then $\alpha_h, \alpha_{i_2}, \dots, \alpha_{i_d}$ are linearly independent for any distinct $i_2, \dots, i_d \in I \setminus \{j\}$.
  If $\alpha_h$ and $\alpha_j$ are not proportional, then $a_{j'} \ne 0$ for some $j' \in I \setminus \{j\}$.
  There exist $d-1$ distinct indices $i_2, \dots, i_d \in I \setminus \{j, j'\}$ since $d \le n - 1$.
  In both cases, we have by Lemma \ref{lem-same-digraph-2} that
  \[\zeta_{j, i_2, \dots, i_d} y_{jh} = \zeta_{h, i_2, \dots, i_d} x_{jh}.\]
  By Lemma \ref{lem-W-map} we also have
  \[\zeta_{j, i_2, \dots, i_d} a_j = \zeta_{h, i_2, \dots, i_d} b_j.\]
  Therefore $x_{jh} a_j = y_{jh} b_j$ as desired.

  \emph{Case three}: $a_j \ne 0$ and $i \ne j$.
  By the same arguments as in case one, we have $b_j \ne 0$.
  We may further assume $x_{ih} y_{ih} \ne 0$ by Corollary \ref{cor-x-y}, otherwise Equation \eqref{eq-W-map-2} reduces to ``0 = 0''.
  Suppose
  \[i_0, \quad  a_1, \quad i_1, \quad  \dots, \quad i_{p-1}, \quad  a_p, \quad   i_p  = i\]
  is an undirected walk in $G_I$ from $i_0$ to $i$, and
  \[i = i_p, \quad a_{p+1}, \quad i_{p+1}, \quad \dots, \quad i_{q-1}, \quad a_q, \quad i_q = j\]
  is an undirected walk in $G_I$ from $i$ to $j$.
  Their concatenation is an undirected walk from $i_0$ to $j$.
  Note that there exist $d-1$ distinct indices $i_2, \dots, i_d \in I \setminus \{i, j\}$ since $d \le n - 1$.
  Then by definitions of $z_i$ and $z_j$ and Lemma \ref{lem-prod-zij}, we have
  \begin{equation}\label{eq-prop-W-map-1}
    \frac{z_j}{z_i} = z_{i_p i_{p+1}} \cdots z_{i_{q-1} i_q} = \frac{\zeta_{j, i_2, \dots, i_d}}{\zeta_{i, i_2, \dots, i_d}}.
  \end{equation}
  Also note that $\alpha_h, \alpha_i, \alpha_{i_2}, \dots, \alpha_{i_d}$ are linearly independent since $a_j \ne 0$.
  Then by Lemma \ref{lem-same-digraph-2} we have
  \begin{equation}\label{eq-prop-W-map-2}
    \frac{x_{ih}}{y_{ih}} = \frac{\zeta_{i, i_2, \dots, i_d}}{\zeta_{h, i_2, \dots, i_d}}.
  \end{equation}
  Moreover, by Lemma \ref{lem-W-map} we also have
  \begin{equation}\label{eq-prop-W-map-3}
    \frac{a_j}{b_j} = \frac{\zeta_{h, i_2, \dots, i_d}}{\zeta_{j, i_2, \dots, i_d}}.
  \end{equation}
  Multiplying Equations \eqref{eq-prop-W-map-1}, \eqref{eq-prop-W-map-2} and \eqref{eq-prop-W-map-3} together, we obtain
  \[\frac{z_j}{z_i} \cdot \frac{x_{ih}}{y_{ih}} \cdot \frac{a_j}{b_j} = 1.\]
  This is exactly Equation \eqref{eq-W-map-2}.
\end{proof}

By Propositions \ref{prop-W-map-1} and \ref{prop-W-map-2},  $f : V_1 \to V_2$ is a homomorphism of $W$-modules.
We have finished the proof of Theorem \ref{thm-main2}.

\section*{Acknowledgments}

The author is deeply grateful to an anonymous referee for useful suggestions which helped to improve this paper.
The author is supported by the Fundamental Research Funds for the Central Universities.

\section*{Declarations of interest}
The author has no relevant interests to declare.

\bibliographystyle{amsplain}
\bibliography{exterior-powers-II}

\providecommand{\bysame}{\leavevmode\hbox to3em{\hrulefill}\thinspace}
\providecommand{\MR}{\relax\ifhmode\unskip\space\fi MR }
\providecommand{\MRhref}[2]{%
  \href{http://www.ams.org/mathscinet-getitem?mr=#1}{#2}
}
\providecommand{\href}[2]{#2}
\begin{thebibliography}{1}

\bibitem{Bourbaki2002}
Nicolas Bourbaki, \emph{{Lie Groups and Lie Algebras. Chapters 4--6}}, Elements
  of Mathematics, Springer-Verlag, Berlin, 2002, Translated from the 1968
  French original by Andrew Pressley.

\bibitem{CIK71}
Charles~W. Curtis, Nagayoshi Iwahori, and Robert~W. Kilmoyer, \emph{Hecke
  algebras and characters of parabolic type of finite groups with {$(B,
  N)$}-pairs}, Inst. Hautes \'{E}tudes Sci. Publ. Math. (1971), no.~40,
  81--116.

\bibitem{FH91}
William Fulton and Joe Harris, \emph{Representation {T}heory. {A} {F}irst
  {C}ourse}, Graduate Texts in Mathematics, vol. 129, Springer-Verlag, New
  York, 1991.

\bibitem{GP00}
Meinolf Geck and G\"{o}tz Pfeiffer, \emph{{Characters of Finite Coxeter Groups
  and Iwahori-Hecke Algebras}}, London Mathematical Society Monographs. New
  Series, vol.~21, The Clarendon Press, Oxford University Press, New York,
  2000.

\bibitem{Henderson10}
Anthony Henderson, \emph{Exterior powers of the reflection representation in
  the cohomology of {S}pringer fibres}, C. R. Math. Acad. Sci. Paris
  \textbf{348} (2010), no.~19-20, 1055--1058.

\bibitem{Hu23-ext-pow}
Hongsheng Hu, \emph{On exterior powers of reflection representations}, Bull.
  Aust. Math. Soc. \textbf{110} (2024), no.~1, 90--102.

\bibitem{Hu23}
\bysame, \emph{Reflection representations of {C}oxeter groups and homology of
  {C}oxeter graphs}, Algebr. Represent. Theory \textbf{27} (2024), no.~1,
  961--994.

\bibitem{Sommers11}
Eric Sommers, \emph{Exterior powers of the reflection representation in
  {S}pringer theory}, Transform. Groups \textbf{16} (2011), no.~3, 889--911.

\bibitem{Steinberg1968}
Robert Steinberg, \emph{{Endomorphisms of Linear Algebraic Groups}}, Memoirs of
  the American Mathematical Society, vol.~80, American Mathematical Society,
  Providence, R.I., 1968.

\end{thebibliography}

\end{document}